\documentclass[a4paper,12pt]{article}
\usepackage{amsmath, amsthm}
\usepackage{mathtools}
\usepackage{amssymb}
\usepackage{amscd}
\usepackage{epic, eepic}
\usepackage{url}
\usepackage{color}
\usepackage[utf8]{inputenc} 
\usepackage{comment}
\usepackage{todonotes}
\usepackage{graphicx}
\usepackage{epstopdf}
\usepackage{enumerate}
\usepackage{array}
\usepackage{tabularx}
\usepackage{tikz}
\usepackage{enumitem}
\usepackage{tcolorbox}
\usepackage{pdflscape}
\usepackage{svg}

\usepackage{caption}
\usepackage{subcaption}

\newtheorem{theorem}{\bf Theorem}[section]
\newtheorem{lemma}[theorem]{\bf Lemma}
\newtheorem{cor}[theorem]{\bf Corollary}

\newtheorem{problem}[theorem]{\bf Problem}
\newtheorem{prop}[theorem]{\bf Proposition}

\newtheorem{nota}[theorem]{\bf Notation}

\newtheorem{claim}[theorem]{\bf Claim}

\newtheorem{remark}[theorem]{\bf Remark}
\newtheorem{defi}[theorem]{\bf Definition}

\title{A note on internal partitions: the $5$-regular case and beyond}

\date{}

\author{Pál Bärnkopf\thanks{Eötvös Loránd University, Budapest, Hungary.
E-mail: {\tt barpal@cs.bme.hu}}
\and  
Zolt\'an L\'or\'ant Nagy\thanks{MTA--ELTE Geometric and Algebraic Combinatorics Research Group,
  E\"otv\"os Lor\'and University, Budapest, Hungary. The author is supported by the Hungarian Research Grant (NKFI) No. K  120154 and 134953.  	E-mail: {\tt nagyzoli@cs.elte.hu}}
\and 
Zoltán Paulovics\thanks{Eötvös University, Budapest. The author is supported by the European Union, co-financed by the European Social Fund (EFOP-3.6.3-VEKOP-16-2017-00002)). Email: {\tt zoli.paulovics@gmail.com }}
}
  
\begin{document}

\maketitle

\begin{abstract}
  An {\em internal} or {\em friendly} partition of a graph is a partition of the vertex set into two nonempty sets so that every vertex has at least as many neighbours in its own class as in the other one. It has been shown that apart from finitely many counterexamples, every $3$, $4$ or $6$-regular graph has an internal partition. In this note we focus on the $5$-regular case and show  that among the subgraphs of minimum degree  at least $3$, there are some which have small intersection. We also discuss the existence of internal partitions in some families of Cayley graphs, notably we determine all $5$-regular Abelian Cayley graphs which do not have an internal partition.    
\end{abstract}

\section{Introduction}

An {\em internal} or {\em friendly} partition of a graph is a partition of the vertices into two nonempty sets so that every vertex has at least as many neighbours in its own class as in the other one. The problem of finding or showing the existence of internal partitions in graphs has a long history. The same concept was introduced by Gerber and Kobler \cite{gerber2004classes} under the name of {\em satisfactory partitions}, while Kristiansen, Hedetniemi and Hedetniemi \cite{kristiansen2002introduction} considered a related problem on {\em graph alliances}. A survey of Bazgan, Tuza and Vanderpooten \cite{bazgan2010satisfactory} describes early results on the area and discusses the complexity of the problem as well as how to find such partitions. Let us denote by $d_G(v)$ the degree of vertex $v$ in graph $G$. For a set $U\subset V(G)$, $d_U(v)$ denotes the number of neighbors of $v$ in $U$.\\
Stiebitz \cite{stiebitz1996decomposing} proved that for every pair of functions
$a, b : V \rightarrow \mathbb{N}^+$ such that $ d_G(v) \geq a(v) +b(v) + 1 \ \forall v \in V$, there exists a partition of the vertex set 
$V(G) = A\cup B$, such that  $d_A(v) \geq  a(v) \  \forall v \in A$ and  $d_B(v) \geq b(v) \ \forall v \in B$.
This confirms a conjecture of Thomassen \cite{thomassen1983graph} in a strong form. 
Kaneko proved  \cite{kaneko1998decomposition} that if $G$ is triangle-free, then $d_A(v) \geq  a(v) \ \forall v \in A$ and $ d_B(v) \geq b(v)  \ \forall v \in B$ can be satisfied even with  $a, b : V \rightarrow \mathbb{N}^+$ such that $ d_G(v) \geq a(v) +b(v) \ \forall v \in V$. This also implies that triangle-free Eulerian graphs have internal partitions, and reveals that the difficulty of the problem is fairly different for regular graphs having odd or even valency.
The condition $ d_G(v) \geq a(v) +b(v) \ \forall v \in V$ cannot be assumed in general, since there are graphs, e.g. $K_{2n}$ which has no partition satisfying  $d_A(v) \geq  a(v$) $\forall v \in A$ and $ d_B(v) \geq b(v)$ $\forall v \in B$.
Likewise there exist infinitely many graphs having no internal partitions, e.g. $K_{2n}$ and $K_{2n+1,2n+1}$. However, several large classes of graphs have been shown to have internal partitions. Diwan proved \cite{diwan2000decomposing}   that if  a graph of girth at least $5$ has  minimum degree  at least $a+b-1$, then its vertex set has a suitable partition $A\cup B$ with minimum degrees $\delta_{G |_A}\geq a$ and $\delta_{G |_B}\geq b$, on the graph induced by $A$ and $B$, respectively.  Moreover,  Ma and Yang \cite{ma2019decomposing} showed that in the last statement of the 
theorem it suffices to assume that $G$ is $C_4$-free. 
Note however that graphs not having internal partitions do not have a forbidden subgraph characterization \cite{shafique2002satisfactory}.\\
  The main goal of this paper is to make a contribution in the case of regular graphs. DeVos posed the following problem.
  
  \begin{problem}[\cite{devos2009}]\label{mainDeVos}  Is it true that all but finitely many $r$-regular graphs have friendly (internal) partitions?
    \end{problem}

For certain small values of $r$, this was confirmed.

\begin{theorem}[Shafique-Dutton \cite{shafique2002satisfactory}, Ban-Linial \cite{ban2016internal}]\label{kisr} Let $r\in \{3,4,6\}$. Then apart from finitely many counterexamples, all $r$-regular graphs have internal partitions. The list of counterexamples is as follows.
\begin{itemize}
    \item for $r=3$, $K_4$ and $K_{3,3}$ do not have an internal partition \cite{shafique2002satisfactory}.
    \item for $r=4$, $K_5$  does not have an internal partition \cite{shafique2002satisfactory}.
    \item for $r=6$, every graph on at least $12$ vertices has an internal partition, thus counterexamples have at most $11$ vertices (and this bound is tight) \cite{ban2016internal}.
\end{itemize}
\end{theorem}
In fact,  Shafique and Dutton conjectured that in the $r$ even case only the complete bipartite graph does not admit an internal partition but this was disproved by Ban and Linial \cite{ban2016internal} who constructed $2k$-regular graphs  on $3k+2$ vertices which does not have such partitions.

There are several directions in which partial results have been achieved recently concerning Problem \ref{mainDeVos}.  A natural weakening of the requirement is to show that a typical, i.e. randomly chosen $r$-regular graph admits an internal partition. One may also pose some restrictions  to obtain an affirmative answer for a large  class of graphs. Another variant is to allow a small proportion of the vertices to have fewer neighbors than required. In these directions impressive 
breakthrough results have been achieved lately.

Linial and Louis proved  \cite{linial2020asymptotically} that for every positive
integer $r$, asymptotically almost every $2r$-regular graph has an internal partition.
Very recently, Ferber et al. resolved  \cite{ferber2021friendly} a conjecture of Füredi  by proving that with high probability, the
random graph $G(n, 1/2)$ admits a partition of its vertex
set into two parts whose sizes differ by at most one in which $n-o(n)$ vertices have at least as many neighbours in their own part as across.

\bigskip

We propose a new direction in the spirit of a lemma of Ban and Linial. For short they introduced the term $k$-cohesive for vertex sets spanning a graph of minimum degree at least $k$.

\begin{prop}[Ban, Linial \cite{ban2016internal}]\label{banlinial}
Every $n$-vertex $d$-regular graph has a $\lceil d/2\rceil$-cohesive set of size at most $\lceil n/2\rceil$ for $d$ even  and of size $n/2+1$ for $d$ odd.
\end{prop}

Problem \ref{mainDeVos} aims for two disjoint $\lceil d/2\rceil$-cohesive sets $A$ and $B$ in $d$-regular graphs provided that $n$ is large enough. Indeed, let us add the vertices from the complement of $A\cup B$ one by one to $A$, provided that they have at least 
$\lceil d/2\rceil$ neighbours in $A$. After the procedure stops, we add the remaining vertices to $B$ and it is straightforward that the resulting partition is internal. \\
Since there are $d$-regular graphs without two disjoint $\lceil d/2\rceil$-cohesive sets, it is natural goal to obtain a good universal upper bound on the intersection size of well chosen pairs of $\lceil d/2\rceil$-cohesive sets
in $d$-regular graphs. This leads to

\begin{problem} Let $\mathcal{G}_{n,d}$ denote the the set of $d$-regular  $n$-vertex graphs.
Determine 
 $$\Phi(n,d):=\max_{G\in \mathcal{G}_{n,d} }\min\left\{\frac{ |V(H_1)\cap V(H_2)|}{n} \ : \ H_i \subseteq G, \delta(H_i)\geq \lceil d/2\rceil \ \forall i\in \{1,2\} \right\}.$$
\end{problem}

\noindent If the answer for Problem \ref{mainDeVos} is affirmative, then clearly $\Phi(n,d)=0$ for fixed $d$ and $n>n_0(d)$. Note also that for $G=K_{d+1}$ and $G=K_{d,d}$, $d$ odd, the intersection size is at least $\frac{2}{d+1}|V(G)|$ and $\frac{1}{d}|V(G)|$, respectively. On the other hand, $\Phi(n,d)=0$ due to Theorem \ref{kisr} for $d\in\{3,4,6\}$ if $n\geq 12$.

Our main result is
\begin{theorem}\label{main}
$\Phi(n,5)\leq 0.2456+o(1)$.
\end{theorem}
We also show a slightly weaker statement, which on the other hand provides an exact result: in each $n$-vertex  $5$-regular graph, the minimum intersection of $3$-cohesive sets is at most $n/4+1$.

We also prove that there are exactly three Cayley graph of valency $5$ over finite Abelian groups which do not admit an internal partition.


\bigskip

Our paper is organized as follows. In Section 2. we briefly summarize the main definitions and notations that we will use throughout the paper and state some results which will serve as a starting point. Then we discuss how results concerning the bisection width relate to our problem  and point out that random-like or expander-like graphs are those in which internal partitions are hard to find. Indeed, as Bazgan, Tuza and Vanderpooten remark \cite{bazgan2010satisfactory}, one can find an internal partition  via a simple local vertex-switching algorithm if there is a bisection of size at most $n/2$. On the other hand, the bisection width of almost all $d$-regular graphs of order $n$ is at least $n(\frac{d}{4} -\frac{\sqrt{d}\ln 2}{2}) $ according to the bound of Bollobás \cite{bollobas1988isoperimetric} and in fact this lower bound bisection  is not far from the upper bound $n\frac{d}{4} -\Theta(n\sqrt{d}) $ due to Alon \cite{alon1997edge}.\\ 
In Section 3.  we prove the main result in a slightly weaker form first. Then we slightly improve the theorem of  Kostochka and Melnikov on the bisection width of sparse graph, making it applicable to non-regular graphs as well.  This enables us to prove the main result of the paper Theorem \ref{main}.
We also discuss a different approach which relies on finding dense enough subgraphs width maximum degree constraints, which may be of independent interest. \\
Motivated by the expander-like property of graphs not having internal partitions, we study some families of Cayley graphs  in Section 4. and characterize those graphs in these families that do not admit such partitions, including the $5$-regular Cayley graphs over finite Abelian groups.  Finally we discuss further open problems in the area in the last section.

\section{Preliminaries and connections to bisection width}

We begin this section by setting the main notations and definitions. Then we discuss the connection between the existence of internal partitions and the minimum size of bisection.

A bisection and a near-bisection of a graph with $n$ vertices is a partition of its vertices into two sets whose  sizes  are the same, and whose sizes differ by at most one, respectively. The bisection size is the number of edges connecting the two sets. Note that finding the bisection of minimum size, in other words, the {\em bisection width}  is NP-hard and only very weak approximations are known in general (see e.g. \cite{feige2000approximating}).

Consider a regular graph $G$ on vertex set $V$. A set $U \subseteq V$ is $k$-cohesive if $G$ restricted to $U$ has minimum degree at least $k$. $G|_U$ denotes the the graph induced by the subset $U$. $N(v)$ denotes the set of neighbors of vertex $v$ while $d(v)$ denotes the degree of vertex $v$, i.e., $d(v)=|N(v)|$. If we consider the degrees w.r.t a certain induced subgraph or another graph the respective graph is indicated in the index.  $N[v]$ denotes the closed neighborhood $N[v]:=N(v)\cup \{v\}$. 

\begin{nota}
We use the notion $n_{(k)}=\binom{n}{k}\cdot k!$ for the falling factorial.
\end{nota}

Let us introduce  two lemmas from the paper of Ban and Linial \cite{ban2016internal}.

\begin{claim}[\cite{ban2016internal}]
\label{(n-3)-regular}
An $(n-3)$-regular graph $G$ has an internal partition if and only if its complementary graph $G$ has at most one odd cycle. Furthermore this partition is a near-bisection.
\end{claim}

\begin{claim}[\cite{ban2016internal}]
For even $n$, every $(n-2)$-regular graph has an internal bisection.
\end{claim}

As it was mentioned in \cite{bazgan2010satisfactory},  relatively small cuts gives evidence to the existence of internal partitions. We present the proof for the case of $5$-regular graphs as this class is the focus of our work and opt to extend it to the case of arbitrary regular graphs as next step.

\begin{prop}
\label{n/2}
If there exists a bisection of a $5$-regular simple graph $G$  of size at most $n/2+5$, then there exists an internal partition for $G$.
\end{prop}

\begin{proof}
Let us call a vertex {\em bad} if it has less neighbors in its own partition class than in the other one. Let us successively move bad vertices from their class to the other class. The number of edges is decreasing between the partition classes in each move. If no bad vertices remain after some moves, we ended up at an internal partition or one of the partition classes became empty. However, the latter case cannot happen. Suppose that on the contrary, one of the classes could became empty at the end of the procedure. Then after at least $(n/2-2)$ moves, we reached a phase where one of the partition classes has size $2$. The number of edges between the two classes is at most $n/2+5-(n/2-2)=7$ at this point, but this contradicts to the valency of~$G$.
\end{proof}
Proposition \ref{n/2} is sharp as the bound $n/2+5$ cannot be improved  according to the result below.

\begin{prop}\label{5reg}
For every even $n\geq 8$, there exists a $5$-regular graph admitting a bisection of size $n/2+6$ in which the algorithm that successively put bad vertices to the other partition class ends with a trivial partition (consisting of the whole vertex set and an empty set).
\end{prop}

\begin{proof} Let $V(G):=\{u_1\ldots u_n\}\cup \{w_1, \ldots w_n\}$, while $E(G):=\{u_iu_{i+1}, w_iw_{i+1} : i\in 1\ldots n-1\}$ 
$\cup  \{u_iw_{i} : i\in 1\ldots n\}\cup \{u_iu_{i+2}, w_iw_{i+2} : i\in 1\ldots n-2\}\cup$\\
$\{u_1w_2, u_1w_n, u_2w_1, u_{n-1}w_n, u_nw_1, u_nw_{n-1}\}$, see Figure \ref{UW}.  The bisection is $U\cup W$.
After  moving  $u_1$ to $W$ as a first step, a chain of moves begins with moving $u_i$ to $W$ in the $i$th step.
\end{proof}

\begin{figure}[!ht]
\centering
  \includegraphics[width=.6\linewidth]{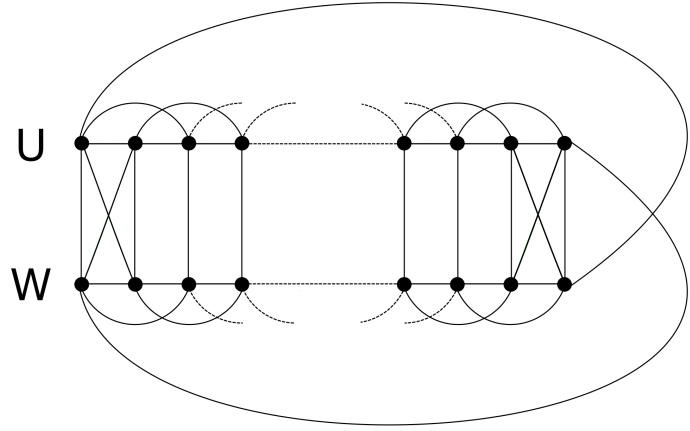}
  \caption{A $5$-regular graph with a relatively small bisection $(U,W)$ where the local switching algorithm fails.}
  \label{UW}
\end{figure}

\begin{theorem}
\label{cor}
\hfill
   \begin{enumerate}[label={\thecor.\arabic*}]

\item If there exists a bisection in a $(2k+1)$-regular  graph $G$  of size at most $n/2+k(k+1)-1$, then there exists an internal partition for $G$.\label{n/2+k(k+1)-1}

\item If there exists a bisection in a $2k$-regular  graph $G$  of size at most $n+k(k-1)-1$, then there exists an internal partition for $G$.\label{n+k(k-1)-1}
\end{enumerate}
\end{theorem}

\begin{remark}\label{sharp}
For every integer $k>0$ and even $n\geq 4k$, there exists a $2k+1$-regular graph admitting a bisection of size $n/2+k(k+1)$ and a  $2k$-regular graph admitting a bisection of size $n+k(k-1)$ in which the algorithm that successively put bad vertices to the other partition class ends with a trivial partition (consisting of the whole vertex set and an empty set).
\end{remark}

\begin{proof}[Proof of Theorem \ref{n/2+k(k+1)-1}] 
We follow the proof of case $k=2$. After at least $(n/2-k)$ moves, we reached a phase where one of the partition classes has size exactly $k$. The number of edges between the two classes is at most $n/2+k(k+1)-1-(n/2-k)=k(k+2)-1$ at this point. This contradicts to the valency of $G$ since at most $\binom{k}{2}$ edges are induced by $k$ points, thus there should be at least $k(2k+1)-2\binom{k}{2}$ edges going between the two sets. 
\end{proof}
\begin{proof}[Proof of Theorem \ref{n+k(k-1)-1}]
We follow the spirit of the proof of case $k=2$. After at least $(n/2-k)$ moves, we reached a phase where one of the partition classes has size $k$. Since each move decreases the number of edges between the partition classes by two, the number of edges between the two classes is at most $n+k(k-1)-1-2(n/2-k)=k(k+1)-1$ at this point, but this contradicts to the valency of $G$. Indeed,  at most $\binom{k}{2}$ edges are induced by $k$ points, thus there should be at least $k\cdot2k-2\binom{k}{2}$ edges going between the two sets. 
\end{proof}

\begin{proof}[Proof of Remark  \ref{sharp}, odd valency] Let $V(G):=\{u_1\ldots u_n\}\cup \{w_1, \ldots w_n\}$, and let $\{u_{i}u_{j} : 1\leq i, j \leq n, 0 < \vert i-j \vert \leq~k \}\cup \{w_{i}w_{j} : 1\leq i, j \leq n, 0 < \vert i-j \vert \leq k \}\cup \{u_1w_n, u_nw_1\}$ be part of the set of edges. In order to obtain a regular graph of valency $2k+1$ we complete the edge set which is possible due to the Gale–Ryser theorem (see \cite{ryser1963combinatorial}, chapter 6) on solving the bipartite realization problem.
Consider the  bisection  $U\cup W$.
After  moving  $u_1$ to $W$ as a first step, a chain of moves begins with moving $u_i$ to $W$ in the $i$th step.
\end{proof}
\begin{proof}[Proof of Remark  \ref{sharp}, even valency] Let $V(G):=\{u_1\ldots u_n\}\cup \{w_1, \ldots w_n\}$, and let $\{u_{i}u_{j} : 1\leq i, j \leq n, 0 < \vert i-j \vert \leq k-1 \}\cup \{w_{i}w_{j} : 1\leq i, j \leq n, 0 < \vert i-j \vert \leq k-1 \}\cup \{u_1w_n, u_nw_1\}$ be part of the set of edges. In order to obtain a regular graph of valency $2k$ we complete the edge set which is possible due to the Gale–Ryser theorem (see \cite{ryser1963combinatorial}, chapter 6) on solving the bipartite realization problem.
Consider the  bisection  $U\cup W$.
After  moving  $u_1$ to $W$ as a first step, a chain of moves begins with moving $u_i$ to $W$ in the $i$th step.
\end{proof}


A result of Diaz, Serma and Wormald \cite{diaz2007bounds} prove that in fact, random $5$-regular graphs indeed have small bisection width. Almost the same bound was obtain by Lyons \cite{lyons2017factors} via a different method, namely using local algorithms.

\begin{theorem}[Diaz, Serma and Wormald \cite{diaz2007bounds}]\label{DSW}
The bisection width of random $5$-regular graphs is asymptotically almost surely below $0.5028n$.
\end{theorem}

As a consequence, 
we note that a tiny improvement on the result of Diaz, Serma and Wormald would imply the existence of internal partitions for almost all $5$-regular graphs, in view of Proposition \ref{5reg}.


\section{Finding cohesive sets with small intersection}

Erdős, Faudree, Rousseau and Schelp  proved the following.

\begin{theorem}[Erdős, Faudree, Rousseau and Schelp  \cite{ErdosFaudree}]\label{EFRS}
Every graph $G$  on  $n\geq k-1$ vertices with at least $(k-1)n-\binom{k}{2}+1$  edges contains a subgraph with minimum degree at least $k$.
\end{theorem}  

\begin{cor}
Specializing to $k=3$, this yields that $n$-vertex graphs on $2n-2$ edges have $3$-cohesive sets.
\end{cor}

This result has been strengthened in the following two directions.

\begin{theorem}[Alon, Friedland and Kalai \cite{alon1984regular}]
\label{AFK}
Let $p$ be a prime power and  $G$ be a graph having average degree  $\Bar{d} > 2p-2$  and maximum degree $\Delta(G) \leq 2p-1$. The  $G$ has a  $p$-regular subgraph.\end{theorem}

This celebrated result  was obtained by a clever application of the Combinatorial Nullstellensatz. Observe that for $k=p$ a dense enough graph $G$ contains not only a $k$-cohesive set but also a subgraph which is $k$-regular.

Sauermann recently proved the following strengthening of the theorem of Erdős et al.

\begin{theorem}[Sauermann \cite{Sauermann}]
\label{Lisa} For every $k$ there exists an $\varepsilon:=\varepsilon_k>0$ such that for
every graph $G$  on  $n$ vertices with at least $(k-1)n-\binom{k}{2}+2$  edges contains a subgraph on at most $(1-\varepsilon)n$ vertices with minimum degree at least $k$.
\end{theorem}

\begin{remark}
Note that these results imply that if one finds a small enough $k$-cohesive set $U$ in a $2k-1$-regular graph, then the result of Erdős, Faudree, Rousseau and Schelp  is applicable to $G\setminus G|_U$. 
\end{remark}

In order to proof our main result, the strategy is similar. Once we obtain a $k$-cohesive set $U$ of minimum size in a $2k-1$-regular  graph $G=G(V,E)$,  we wish to delete a set $E^*$ of edges such that 
\begin{itemize}
    \item $|E\setminus E^*|\geq (k-1)n-\binom{k}{2}+1$ and 
    \item $G^*(V,E^*)$ has as many vertices $v\in U$ of degree at least $k$ as possible.
\end{itemize}
This would in turn 
provide a pair of $k$-cohesive sets with small intersection, due to Theorem \ref{EFRS}. We discuss further only the case $k=3$, however the methods below can be generalized.

\subsection{Proof of the main result}

First we reiterate the lemma of Ban and Linial.
\begin{prop}[Ban, Linial \cite{ban2016internal}]
Every $n$-vertex $d$-regular graph has a $\lceil d/2\rceil$-cohesive set of size at most $\lceil n/2\rceil$ for $d$ even  and of size at most $n/2+1$ for $d$ odd.
\end{prop}

We consider a result which  may count on independent interest as well. The problem is to find a subgraph of fixed order with maximum number of edges which fulfills an extra constraint on a maximum degree. Some related work can be found  in \cite{griggs1998extremal, furedi2002turan}. 

\begin{prop}\label{minfokos}
If $H$ is  a $3$-cohesive  graph on $n$ vertices with maximum degree $5$, then for each $1\leq k\leq n$ there exists a subgraph $H'$ such that $|V(H')|=k$, $|E(H')|\geq k-1$ and the maximum degree $\Delta(H')$ of $H'$ is at most $\Delta(H')\leq 3$.
\end{prop}

\begin{proof}
 Observe that it is enough to prove this for connected components of $H$, thus we may assume that $H$ is connected. First we show that the statements holds for $k\leq 0.88n$. It clearly does hold for $k=1$. Suppose that by contradiction, there exists a number $k$ less than $|V(H)|$ for which the statement fails for a certain graph $H$ and let us choose the smallest $k$ with that property. Hence we get that for each subgraph $H'\subset H$ on $k-1$ vertices with maximum degree $3$, $|E(H')|\leq k-2$. Indeed, otherwise we could add an isolated vertex to obtain a subgraph on $k$ vertices with the prescribed property. Moreover since $k$ is the smallest such number, there exists a subgraph $H'$ on $k-1$ vertices with maximum degree $3$, $|E(H')|= k-2$.\\
  Let us consider such a subgraph $H'$ on $k-1$ vertices and  maximum number of edges, and denote by $t_i$ the number of vertices of degree $i$ in $H'$. We in turn obtain that 
  
  \begin{equation}
      t_0+t_1+t_2+t_3 = |V(H')|=k-1.
  \end{equation}
  
    \begin{equation}\label{eq_2}
      t_1+2t_2+3t_3 = 2|E(H')|= 2(k-2)=2(t_0+t_1+t_2+t_3)-2.
  \end{equation}
Consider now the edges in $E(H)\setminus E(H')$. 
Since $H'$ was maximal with respect to the number of edges,  if $uv\in E(H)\setminus E(H') $ and $v\in V(H')$ is of degree $d_{H'}(v)<3$,  then $d_{H'}(u)=3$. Indeed, $d_{H'}(u)<3$  with $u\in V(H')$ contradicts to the maximality of the $H'$ w.r.t. the number of edges, while $u\in V(H)\setminus V(H')$  would imply that $u$ together with the edge $uv$ can be added to $H'$ to obtain a subgraph with the prescribed property.
Thus by double counting the edges from $E(H)\setminus E(H')$
between vertices $v\in V(H')$ having degree  $d_{H'}(v)<3$ and vertices $u\in V(H')$ having degree $d_{H'}(u)=3$, we obtain 

\begin{equation}\label{eq_3}
     3t_0+2t_1+t_2 \leq 2t_3.
  \end{equation}

Here we also used that $3\leq d_{H}(v)\leq 5$ for all $v\in V(H)$.
However, Inequalities (\ref{eq_2}) and (\ref{eq_3}) together yield 
\begin{equation}\label{eq_4}
    3t_0+2t_1+t_2 \leq 2t_3 \leq 4t_0+2t_1-4,
  \end{equation}
and this is in turn a contradiction unless $t_0\geq 4$. The maximality of $H'$ w.r.t. the number of edges however implies also that $t_0\geq 2$ can only occur if there is no  $e\in E(H)\setminus E(H')$ joining two vertices from $V(H)\setminus V(H')$. In other words, each of these vertices must be connected to the set of vertices having degree $d_{H'}(v)=3.$ Hence if $t_0\geq 2$, then Inequality (\ref{eq_4}) can be improved as follows.

\begin{equation}\label{eq_5}
   3(n-k+1)+ 3t_0+2t_1+t_2 \leq 2t_3 \leq 4t_0+2t_1-4.
  \end{equation}

We have $3t_0\leq 2t_3$  from Inequality (\ref{eq_3}), thus $t_0\leq k-1-t_3\leq k-1-1.5t_0$, so we get $t_0\leq \frac{2}{5}(k-1)$.

Putting all together we obtain 
\begin{equation}\label{eq_6}
   3(n-k+1) \leq  t_0-4\leq \frac{2}{5}(k-1)-4,
  \end{equation}
which is a contradiction for $k<\frac{3n+7.4}{3.4}.$\\

In the case $k\geq 0.88n$ we apply the probabilistic method. Let $m_i$ denote the number of vertices of degree $i$ in our $3$-cohesive graph $H$.
We have $n=m_3+m_4+m_5$. Let us choose uniformly at random a set $Z$ of $\lambda n$ distinct vertices from $V(H)$. Moreover, let $X$ denote the random variable counting the number of edges in $Z$. To obtain a suitable edge set, we must  delete an edge from each vertex of degree $4$ and delete a pair of edges from each vertex of degree $5$ in $Z$. (Note that we may suppose that there are no edges joining vertices of degree larger than $3$  in $H$.) Let $Y$ denote the random variable which counts the number of edges which we should delete to obtain a graph on $Z$ of maximum degree $3$. Then we have

\begin{equation}
    \begin{split}
\mathbb{E}(X-Y)=\sum_{e\in E(H)}\mathbb{I}(\{x,y\}\subset Z : xy=e) - \sum_{v\in V(H), d(v)=4}\mathbb{I}(N[v]\subset Z ) \\- \sum_{v\in V(H), d(v)=5}2\mathbb{I}(N[v]\subset Z ) - \sum_{v\in V(H), d(v)=5}\mathbb{I}(v\in Z, |N(v)\cap Z|=4  ).
    \end{split}
\end{equation}
Calculating the expressions above we obtain 

\begin{equation}
    \begin{split}
\mathbb{E}(X-Y)&=\frac{1}{2}(3m_3+4m_4+5m_5)\frac{\binom{\lambda n}{2}}{\binom{ n}{2}} - m_4\frac{\binom{\lambda n}{5}}{\binom{ n}{5}} - 2m_5\frac{\binom{\lambda n}{6}}{\binom{ n}{6}} - 5m_5\frac{{(\lambda n)}_{(5)}(1-\lambda)n}{{ n}_{(6)}}\\&\geq
\frac{3}{2}\lambda^2\cdot n+m_4\left(\frac{\lambda^2}{2} -\frac{\binom{\lambda n}{5}}{\binom{ n}{5}}\right)+m_5\left({\lambda^2}-2\frac{\binom{\lambda n}{6}}{\binom{ n}{6}}-5\frac{{(\lambda n)}_{(5)}(1-\lambda)n}{{ n}_{(6)}}\right)\\&\geq
\frac{3}{2}\lambda^2\cdot n+m_4\left(\frac{\lambda^2}{2}-\lambda^{5}\right)+m_5\left({\lambda^2}-2{\lambda }^{6}-5\lambda ^{5}(1-\lambda)\frac{n}{n-5}\right) .
    \end{split}
\end{equation}
Suppose that $n\geq 15$ and $\lambda\geq 0.8$. Then both $m_4$ and $m_5$ have negative coefficient, moreover, their ration is smaller than $7/8$. This means that the minimum of the expression with respect to the inequality $3m_3\geq 4m_4+5m_5$ takes its value when $m_4=0$ and $m_5=\frac{3}{8}n$. However, 

\begin{equation}
    \begin{split}
\frac{3}{2}\lambda^2\cdot n+\frac{3}{8}n\left({\lambda^2}-2{\lambda }^{6}-5\lambda ^{5}(1-\lambda)\frac{n}{n-5}\right) >\lambda \cdot n,
    \end{split}
\end{equation}
thus there exists a subgraph of size $k=\lambda\cdot n$ with at least $k$ edges and each vertex has degree at most $3$.
\end{proof}

Note that the constraint on the maximum degree of $G$ was essential to obtain a linear bound on the edge cardinality. Indeed, a biregular complete bipartite graph with one class consisting of vertices of degree $3$ shows that if one omits that constraint, only a constant number of edges can be guaranteed in the subgraphs for each order.

Now we are ready to prove the weaker form of our main result.

\begin{theorem}\label{n/4}
Suppose that $G$ is a $5$-regular graph on $n$ vertices. Then there are  two distinct internal sets $V_1,V_2 \subset V(G)$ such that $\vert V_1 \cap V_2 \vert \leq n/4+1.$ 
\end{theorem}

\begin{proof}
Due to Proposition \ref{banlinial} we have a $3$-cohesive subgraph $H\subset G$ on at most $n/2 + 1$ vertices. Our goal is to determine an edge set $E^*$ of size at most $n/2+2$ such that  after deleting it we can use Theorem \ref{EFRS} to find another $3$-cohesive set with a common intersection of size at most $n/4+1$.

First we use Lemma \ref{minfokos}  in order to find a subgraph $H'\subset H$  such that $|V(H')|= n/4$, $|E(H')|\geq n/4-1$ and the maximum degree $\Delta(H')$ of $H'$ is at most $\Delta(H')\leq 3$.  It is easy to see that we can add $t$ edges to $H'$ from $E(H)$ to increase the degree of each vertex to at least $3$, such that $t\leq n/4+2$. Thus we obtain an edge set of cardinality  $\vert E^* \vert \leq  n/4-1+n/4+2=n/2+1.$

Finally we  apply Theorem \ref{EFRS}  to the graph  obtained by deleting the edges of $E^*$ from $G$,  which yields  a $3$-cohesive subgraph $G'$ with $\vert V(G') \cap V(H) \vert \leq n/4+1.$  
\end{proof}

\subsection{Improvement via  the result of Kostochka and Melnikov}

In order to improve Theorem \ref{n/4}, our aim is to strengthen Proposition \ref{minfokos} by pointing out the existence of a denser subgraph with the same constraints on the maximum degree.
We proceed by applying a generalized version of a theorem of Kostochka and Melnikov.

\begin{theorem}[Kostochka, Melnikov, \cite{kostochka1992bounds}]
 For any given natural number $d\geq 2$ and for any connected $d$-regular graph $G$ on $n$ vertices, the bisection width $bw(G)$ fulfils
$$bw(G)\leq \frac{d-2}{4}n+ O(d\sqrt{n}\log{n}).$$
\end{theorem}

The proof consists of two main steps. First the authors cluster the vertex set of the graph to even number of
 equal classes  (of size roughly  $\sqrt{n}$), apart from a small set of remainder vertices, in such a way that all the clusters contain at least roughly $\sqrt{n}$ edges. Then they randomly distribute the classes into two large cluster of equal size, and they do the same with the remainder vertices as well. It is easy to verify that the  generalization below also follows from their proof.

\begin{theorem}[A generalization of the  Kostochka--Melnikov bound]\label{gen_Kost_bw}
 For any given rational number $d\geq 2$, positive constant $c\geq 1$ and for any $n$-vertex connected  graph $G$ of average degree $d$ and maximum degree at most $c$d, its bisection width fulfils
$$bw(G)\leq \frac{d-2}{4}n+ O(d\sqrt{n}\log{n}).$$
\end{theorem}

If we are interested in a dense subgraph of given order, the same approach provides the following bound.

\begin{theorem}
\label{gen_Kost}
 Let $d\geq 2$ be a rational number $d\geq 2$, and $\alpha\in(0,1) $, $c\geq 1$ positive constants.
 For any $n$-vertex  connected  graph $G$ of average degree $d$ and maximum degree at most $c$d, there exists  a subgraph $G'$ on $\lfloor \alpha n\rfloor$ vertices, which have at least $(\alpha+o(1)) n+\frac{d-2}{2}\alpha^2n$ edges.
\end{theorem}

Let $\mu\in (0,1)$ be  the real root  of $36x^5 -45 x^4 +8$. Note that  $\mu \approx 0.88.$ Now we are ready to make an improvement on Proposition \ref{minfokos}.

\begin{prop}\label{minfof(k)os}
If $H$ is  a $3$-cohesive  graph on $n$ vertices with maximum degree $5$, then for each $0 \leq k \leq n$ there exists a subgraph $H'$ such that $|V(H')|=k$, $|E(H')|\geq f(k)$ and  $\Delta(H')\leq 3$, where

\begin{equation*}
f(k) = \begin{cases}
             k+0.1355 k^2/n  & \text{if } k \leq \mu n \\
             1.875k^2/n - 1.875 k^5/n^4 + 1.125 k^6/n^5  & \text{if } k > \mu n.
       \end{cases} \quad
\end{equation*}

\end{prop}

\begin{proof} We apply the probabilistic method of Proposition \ref{minfokos} together with a probabilistic clustering of the graph in the spirit of the Kostochka-Melnikov bound.
Let $m_i$ denote again the number of vertices of degree $i$ in our $3$-cohesive graph $H$, which implies $n=m_3+m_4+m_5$. We may suppose that each edge is incident to a vertex of degree $3$, otherwise erasing the edge would still result a $3$-cohesive graph. Let us choose uniformly at random a set $Z$ of $c_1 n$ distinct vertices from $V(H)$. The constant $c_1=c_1(k)$ is chosen later on in order to obtain an optimized bound. Let $X$ denote the random variable counting the number of edges in $H|_Z$. To obtain a suitable subgraph with maximum degree at most $3$, we must delete an edge from each vertex of degree $4$ and delete a pair of edges from each vertex of degree $5$ in $H|_Z$. Let $Y$ denote the random variable which counts the number of edges which we should delete to obtain a graph on $Z$ of maximum degree $3$ as described above. Then we have

\begin{equation}
    \begin{split}
\mathbb{E}(X-Y)=\sum_{e\in E(H)}\mathbb{I}(\{x,y\}\subset Z : xy=e) - \sum_{v\in V(H), d(v)=4}\mathbb{I}(N[v]\subset Z ) \\- \sum_{v\in V(H), d(v)=5}2\mathbb{I}(N[v]\subset Z ) - \sum_{v\in V(H), d(v)=5}\mathbb{I}(v\in Z, |N(v)\cap Z|=4  ).
    \end{split}
\end{equation}
Calculating the expressions above we obtain 

\begin{equation}\label{E(x-y)}
    \begin{split}
\mathbb{E}(X-Y)&\geq
\frac{3}{2}c_1^2\cdot n+m_4\left(\frac{c_1^2}{2}-c_1^{5}\right)+m_5\left({c_1^2}-2{c_1 }^{6}-5c_1 ^{5}(1-c_1)\frac{n}{n-5}\right).
    \end{split}
\end{equation}
This implies the existence of a dense enough subgraph on a set $V_1$ of $c_1n$ vertices, which has maximum degree at most $3$.

Now, we use Theorem \ref{gen_Kost} to find a set $V_2\subseteq V_1$ with  $\vert V_2 \vert =k= c_2 n$, i.e., $\alpha=c_2/c_1$.  Then $e:= e(G[V_2]) \geq (c_2+o(1)) n + (\mathbb{E}(X-Y) - c_1 n) \big( \frac{c_2}{c_1} \big)^2.$ Thus we get 
\begin{equation}\label{eq_kifejezes}
\begin{split}
\frac{e}{c_2 n} & \geq (1+o(1)) +\frac{\mathbb{E}(X-Y) c_2}{c_1^2n } - \frac{c_2}{c_1}\\ &\geq   
(1+o(1)) +
\frac{c_2}{c_1^2n }\left(\frac{3}{2}c_1^2\cdot n+m_4\left(\frac{c_1^2}{2}-c_1^{5}\right)+m_5\left({c_1^2}-2{c_1 }^{6}-5c_1 ^{5}(1-c_1)\right)\right)-\frac{c_2}{c_1}.
\end{split}
\end{equation}

Our aim is to determine $c_1=c_1(k)$ which provides the best universal lower bound for the right hand side of (\ref{eq_kifejezes}). In order to do this, we have to find the extremum with restrictions $n=m_3+m_4+m_5$, $m_i\geq 0$. 
We know that the extremum is admitted at a point where at least one of the variables $m_3, m_4, m_5$ equals to zero.  Furthermore, $m_4\leq \frac{3}{7}n$ and $m_5\leq \frac{3}{8}n$ since vertices of degree $5$ are joint to vertices of degree $3$ according to our assumption.\\

\textbf{ Case 1.} $m_5 = 0$ and $m_4 = \lambda n$, $\lambda\in [0, 3/7].$

$$\frac{e}{c_2 n} \geq 1+ \frac{3 c_2}{2} - \frac{c_2}{c_1} + \frac{m_4 c_2}{n} (\frac{1}{2} - c_1^3) = 1+ \frac{3 c_2}{2} - \frac{c_2}{c_1} + \lambda c_2(\frac{1}{2} - c_1^3).$$

This expression is linear in $\lambda$, so the minimum is taken at $0$ or $\frac{3n}{7}.$

If $\lambda = 0$, then $c_1=1$ is the best choice, which yields the lower bound $1+ \frac{c_2}{2}$. If $\lambda = \frac{3}{7}$, then the maximum of $1+ \frac{3 c_2}{2} - \frac{c_2}{c_1} + \frac{3 c_2}{7}(\frac{1}{2} - c_1^3)$ is at $c_1 = 1$, hence the expression is monotonically increasing between $0$ and $\sqrt[3]{\frac{7}{6}}.$ Thus the minimum is $1+\frac{2c_2}{7}$, and while $1+\frac{2c_2}{7} \leq 1+ \frac{c_2}{2}$ in this case the minimum is $1+\frac{2c_2}{7}.$\\

Therefore $e \geq (1+\frac{2c_2}{7})c_2 n.$\\

\textbf{Case 2.} $m_4 = 0$ and $m_5 = \lambda n,$ $\lambda\in [0, 3/8].$

$$\frac{e}{c_2 n} \geq 1+ \frac{3 c_2}{2} - \frac{c_2}{c_1} + \frac{m_5 c_2}{n} (1 -5c_1^3 + 3c_1^4) = 1+ \frac{3 c_2}{2} - \frac{c_2}{c_1} + \lambda c_2 (1 -5c_1^3 + 3c_1^4).$$

Let $f(c_2, c_1, \lambda)$ denote the expression on the right hand side. Since it is linear in $\lambda$,  the minimum is at $0$ or $\frac{3n}{8}.$

For $\lambda = 0$  we get back again the bound  $1+ \frac{c_2}{2}$.
For $\lambda = \frac{3n}{8}$,  we determine the maximum value of $1+ \frac{3 c_2}{2} - \frac{c_2}{c_1} + \frac{3 c_2}{8} (1 -5c_1^3 + 3c_1^4)$ with partial differentiation:

$$\frac{d (1+ \frac{3 c_2}{2} - \frac{c_2}{c_1} + \frac{3 c_2}{8} (1 -5c_1^3 + 3c_1^4))}{dc_1}= \frac{(36 c_1^5 -45 c_1^4 +8)c_2}{8c_1^2}.$$

So the maximum point $\mu$ is the feasible solution of $36 c_1^5 -45 c_1^4 +8$, that is $\mu \approx 0.88.$
In this case the minimum is $f(c_2,  \mu, 3/8)$ for $c_2 \leq \mu$. Otherwise, since $c_2 \leq c_1$ the minimum is $f(c_2,c_2,3/8)$ at $c_1=c_2$.

This concludes to   $e \geq (1+0.1355c_2) c_2 n$ for $c_2 \leq \mu$, and  $e \geq (\frac{15c_2}{8} - \frac{15c_2^4}{8} + \frac{9 c_2^5}{8})c_2 n= (1.875c_2 - 1.875 c_2^4 + 1.125c_2^5) c_2 n$ for $c_2 > \mu$ in Case 2.\\

Finally, by comparing the results of Case 1 and Case 2, we have the following.
If $c_2 \leq \mu$ then the minimum of $\frac{e}{c_2n}$ is $f(c_2,\mu, 3/8)$,  which is approximately $1+0.1355c_2.$
Otherwise it is  $f(c_2, c_2, 3/8)$, which gives the expression  $1.875c_2 - 1.875 c_2^4 + 1.125c_2^5$.
\end{proof}

\begin{theorem}
Suppose that $G$ is a $5$-regular graph on $n$ vertices. Then there are two distinct internal sets $A,B \subset V(G)$ such that $\vert A \cap B\vert \leq (0.2456+o(1)) n$
\end{theorem}

\begin{proof}
We follow  the proof of Theorem \ref{n/4}.
Due to Proposition \ref{banlinial} we have a $3$-cohesive subgraph $H\subset G$ on at most $n/2 + 1$ vertices.

First we use Lemma \ref{minfof(k)os}  in order to find a subgraph $H'\subset H$  such that $|V(H')|=~k$, $|E(H')|\geq f(k)$ and  $\Delta(H')\leq 3$. Then we can add $t$ edges to $H'$ from $E(H)$ to increase the degree of each vertex to at least $3$, such that $t\leq 3k-2f(k)$. Thus we obtain an edge set of cardinality  $\vert E^* \vert \leq  3k-f(k).$

To  apply Theorem \ref{EFRS}  to the graph  obtained by deleting the edges of $E^*$ from $G$, we need $3k-f(k)\leq \frac{n}{2}+2$ to hold.
This implies the choice
$$k = \frac{2n-\sqrt{3.729n^2-1.626n}}{2 \cdot 0.1355}\approx 0.2456n+o(n),$$ which satisfies these conditions.
\end{proof}

\section{Internal partitions in Cayley graphs}

As we could see in Section 2, the existence of internal partition follows if the  bisection width is not large, or in general, if there is an almost balanced vertex cut of relatively small size. A celebrated theorem of Bollobás \cite{bollobas1988isoperimetric} proves that  random  $r$-regular graphs provide good expanders in the sense that the isoperimetric number is large compared to $r$, thus these vertex cut sizes are relatively large.
Hence to seek graphs without internal partitions, it is natural to investigate well structured expander graphs.

The first observation is derived by a computer-aided search.
\begin{claim}
There exists an internal partition in every Paley graph of order less than $500$.
\end{claim}

Next we study the existence of internal partitions in $5$-regular Cayley graphs.

\begin{defi}
Let $G$ be a finite group and let S be a subset of G satisfying $0 \not\in S$, and $S = -S$, i.e., $s \in S$ if and only if $-s \in S$. Then define the \emph{Cayley graph} on group $G$ with connection set $S$, denoted $Cay(G; S)$, to have its vertices labelled with the elements of $G$ and $x$ adjacent to $y$ if and only if $y = x + s$ for some $s \in S$.
\end{defi}

\begin{defi}
 $G$ is called an (additive) cyclic Cayley graph with a generating set $(i_1, \dots, i_t)$ if $G=Cay(K,S)$, where $K$ is a cyclic group and $S=\{\pm i_1, \dots, \pm i_t\}$. If $K \simeq \mathbb{Z}_n$, then we denote $G$ by $\langle i_1, \dots, i_t \rangle _n$.
\end{defi}

\subsection{Cyclic Cayley graphs}

\begin{theorem}\label{5regcay}
Every $5$-regular cyclic Cayley-graph has an internal partition except for $K_6, K_{5,5}$, and $\langle 1,2,5 \rangle _{10}$.
\end{theorem}

Observe first that the order  of the group must be even, $n=2k$. Furthermore, if the  cyclic Cayley graph  has odd valency, then $k$ must be one of the generators and we may suppose that other generators are less than $k$.
We begin with some auxiliary lemmas.

\begin{claim}
\label{rel prime}
Suppose that $(r,2k)=1$ holds for positive integers $r,k$. Then $\langle r, t, k \rangle _{2k}$ is isomorphic to $\langle 1, t^*, k \rangle _{2k}$, where $r \cdot t^* \equiv t\pmod {2k}$.
\end{claim}

\begin{proof}
Let $v_1$ be an element of the vertex set of $\langle 1, t^*, k \rangle _{2k}$, $v_1$ is labeled with $g_1 \in (\mathbb{Z}_n,+)$ and $v_2$ an element of the vertex set of $\langle r, t, k \rangle _{2k}$. Let $v_2$ be assigned to $v_1$, if $v_2$ is labeled with $g_2 = r \cdot g_1$. It is a bijection, because $(r,2k)=1$ and $g_1-g_2 \in \{\pm 1, \pm t^*,k\}$ if and only if $r \cdot g_1- r \cdot g_2 = r \cdot (g_1-g_2) \in \{\pm r, \pm t,k\} \pmod{2k}$, therefore it is an isomorphism between $\langle 1, t^*, k \rangle _{2k}$ and $\langle r, t, k \rangle _{2k}$.
\end{proof}


\begin{claim}
\label{oszthato}
If $(t,k) \neq 1$  then $\langle r, t, k \rangle _{2k}$ has an internal partition.
\end{claim}

\begin{proof}
Consider the congruence classes of $\{ 1, 2, \ldots , 2k \}$ modulo $(t,k)$. It is easy to check that they are internal subsets of $\langle r, t, k \rangle _{2k}$: if the distance of two elements is $t$ or $k$ then they will be in the same class, hence every vertex degree is at least $3$.

Since $1< (t,k) <k$ (according to $0<t<k$), thus $\frac{2k}{(t,k)}>2$. Therefore, we find two disjoint internal subsets, which completes the proof.
\end{proof}

\begin{proof}[Proof of Theorem \ref{5regcay}]
If $(r,k) \neq 1$ or $(t,k) \neq 1$ then we are done by Claim \ref{oszthato}. In the remaining case, $r$ and $t$ are even integers or without loss of generality we can assume that $(r,2k)=1$. In the first case the vertices with even index will define an internal partition set. In the second case, by Claim \ref{rel prime} it is enough to examine the graphs $\langle 1,t^*,k \rangle_{2k}$.


First, we assume that $k \geq 8$. 
It is easy to check that $$\{ 1,2,t^*+1,t^*+2,k+1,k+2,t^*+k+1,t^*+k+2 \} \mbox{ \ and \ \ }$$ $$\{ 3,4,t^*+3,t^*+4,k+3,k+4,t^*+k+3,t^*+k+4 \}$$  will be a pair of disjoint internal subsets
for $t^* \in \{ 4, \ldots ,k-4\}$ (see Subfigure 1), and  similarly,
$$\{ 1,2,3,4,k,k+1,k+2,k+3 \}\mbox{ \ and \ \ }\{ 5,6,7,8,k+4,k+5,k+6,k+7 \}$$ will be a pair of disjoint internal subsets for $t^* \in  \{2,3, k-3, k-2, k-1\}$  (see Subfigure 2).

\begin{figure}
\centering
\begin{subfigure}{.5\textwidth}
  \centering
  \includegraphics[width=.8\linewidth]{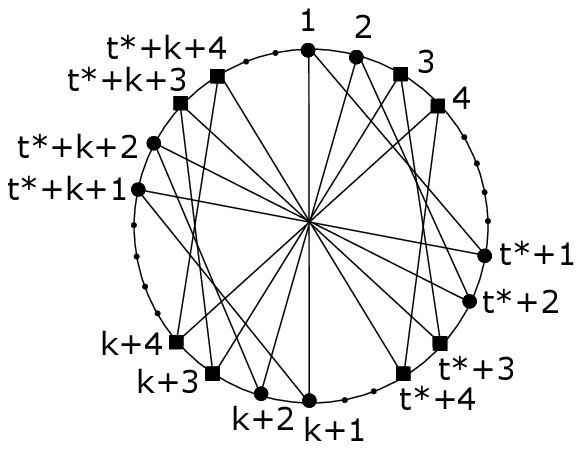}
  \caption{Subfigure 1}
  \label{fig:sub1}
\end{subfigure}%
\begin{subfigure}{.5\textwidth}
  \centering
  \includegraphics[width=.6\linewidth]{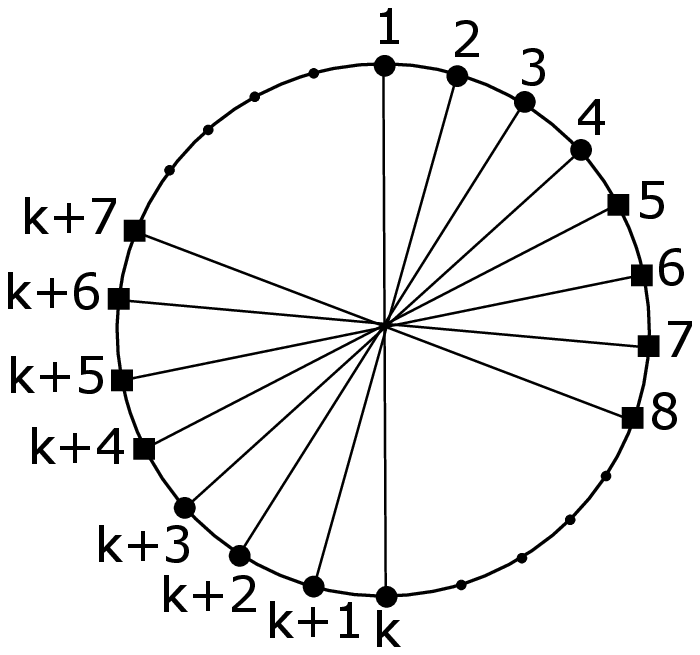}
  \caption{Subfigure 2}
  \label{fig:sub2}
\end{subfigure}
\caption{}
\label{fig:test}
\end{figure}

$\langle 1,2,3 \rangle _{6}$ is $K_6$,
$\langle 1,2,5 \rangle _{10}$ is $P_{2,5}$, 
$\langle 1,3,5 \rangle _{10}$ is $K_{5,5}$, so the  list of Table \ref{tab:cayley} summarizes the remaining cases.\end{proof}


\begin{table}[ht]
\caption{Table of small $5$-reg. Cayley graphs with internal partitions}
\begin{center}
 \begin{tabular}{||c c||}
 \hline
 Example &  Internal sets \\ [0.5ex] 
 \hline\hline

$\langle 1,2,4 \rangle _{8}$ & $\{ 1,3,5,7 \}, \{ 2,4,6,8 \}$ \\ 
 \hline
$\langle 1,3,4 \rangle _{8}$ & $\{ 1,2,5,6 \}, \{ 3,4,7,8 \}$ \\ 
  \hline
$\langle 1,4,5 \rangle _{10}$ & $\{ 1,2,6,7 \}, \{ 3,4,8,9 \}$ \\ 
  \hline
$\langle 1,2,6 \rangle _{12}$ & $\{ 1,2,3,7,8,9 \}, \{ 4,5,6,10,11,12 \}$ \\ 
  \hline
 $\langle 1,3,6 \rangle _{12}$ & $\{ 1,4,7,10 \}, \{ 2,5,8,11 \}$ \\ 
  \hline
$\langle 1,4,6 \rangle _{12}$ & $\{ 1,3,5,7,9,11 \}, \{ 2,4,6,8,10,12 \}$ \\ 
  \hline
$\langle 1,5,6 \rangle _{12}$ & $\{ 1,2,7,8 \}, \{ 3,4,9,10 \}$ \\ 
  \hline
$\langle 1,2,7 \rangle _{14}$ & $\{ 1,2,3,8,9,10 \}, \{ 4,5,6,11,12,13 \}$ \\
 \hline
$\langle 1,3,7 \rangle _{14}$ & $\{ 1,4,5,8,11,12 \}, \{ 3,6,7,10,13,14 \}$ \\
 \hline
$\langle 1,4,7 \rangle _{14}$ & $\{ 1,4,5,8,11,12 \}, \{ 3,6,7,10,13,14 \}$ \\
 \hline
$\langle 1,5,7 \rangle _{14}$ & $\{ 1,2,3,8,9,10 \}, \{ 4,5,6,11,12,13 \}$ \\
 \hline
$\langle 1,6,7 \rangle _{14}$ & $\{ 1,2,3,8,9,10 \}, \{ 4,5,6,11,12,13 \}$ \\  [1ex] 
 \hline
 \end{tabular}\label{tab:cayley}
\end{center}
\end{table}

Based on the cyclic  Cayley graph $P_{2,5}$, it is natural to ask whether there exist cyclic Cayley graphs for each valency $r$, which are different from  $K_{r+1}$ and $K_{r,r}$, furthermore which do not admit an internal partition.

\begin{prop}
For every  even $n>2$ there exists a $(n-3)$-regular cyclic Cayley graph on $n$ vertices which does not contain an internal partition if and only if $n$ is not a power of $2$.

\end{prop}

\begin{proof}
If $n$ is not a power of $2$ then it can be written  of form $n=l \cdot m$, where $l>1$ is odd. Consider the graph $\langle m \rangle _{n}$. It is the union of $m>1$ pieces of cycles of length $l$. Hence by Claim \ref{(n-3)-regular}, there is no internal partition in the complementary of this graph. Therefore we found a $(n-3)$-regular cyclic Cayley graph on $n$ vertices, such that it does not contain an internal partition.

Consider a $(n-3)$-regular cyclic Cayley graph on $n$ vertices, such that it does not contain an internal partition. The complementary of this graph (denoted by $\langle s \rangle _{n}$) is the union of cycles with same the length, and $n$ is divisible by this common length. According to Claim \ref{(n-3)-regular} there is at least $2$ cycles with odd length in $\langle s \rangle _{n}$. Hence $n$ has an odd divisor, therefore $n$ is not a power of $2$.
\end{proof}

\subsection{ Cayley graphs on the group $\mathbb{Z}_{2^t}$}

Let $G=Cay(\mathbb{Z}_2^t;\{g_1, \dots, g_k\})$. Now for all $1 \leq i \leq k$, the edges generated by $g_i$ determine a perfect matching, because all element of $\mathbb{Z}_2^t$ is the negative of himself.

\begin{theorem}
    Let $G=Cay(\mathbb{Z}_2^t;\{g_1, \dots, g_5\})$. Then $G$ has an internal partition.
\end{theorem}

\begin{proof}
    If $t=3$, then the complementary of $G$ is the union of two perfect matching, so it is two-regular and bipartite graph (with the vertex sets $A$ and $B$). $A$ and $B$ induce a $K_4$ in the graph $G$, so they determine an internal partition.
    
    If $t>3$, then consider three generators  $g_1, g_2, g_3$. We can assume that $g_3 \neq g_1+g_2$, otherwise we change $g_3$ and $g_4$. Then $0, g_1, g_2, g_3, g_1+g_2, g_1+g_3, g_2+g_3$ and $g_1+g_2+g_3$ are distinct element and they are connected as shown in  Figure \ref{UW1}.  $G$ thus can be tiled by its subgraph $G'=Cay(\mathbb{Z}_2^t;\{g_1, g_2, g_3\})$, so we found $2^{t-3}$ disjoint internal sets which implies the existence of an internal partition in $G$.
\end{proof}

\begin{figure}[!ht]
  \centering
  \includegraphics[width=.3\linewidth]{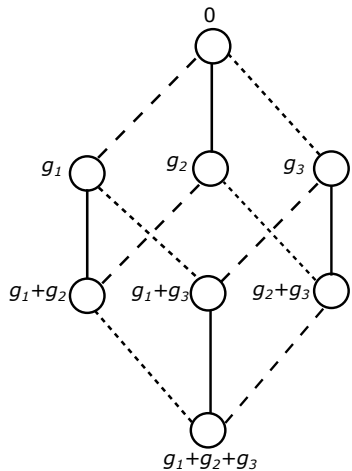}
  \caption{Graph generated by three elements}
  \label{UW1}
\end{figure}

\subsection{Cayley graphs of finite Abelian groups }

We apply the structure theorem of finite Abelian groups and deduce that apart from some Cayley graphs arising from the small cyclic groups, every $5$-regular Cayley graph over a finite Abelian group has an internal partition. First we prove a special case and extend Theorem \ref{5regcay}.

\begin{prop}
\label{5regcay2} 
Suppose that $p>1$ is a positive integer. Then every $5$-regular Cayley graph on the group $\mathbb{Z}_2 \times \mathbb{Z}_{2p}$ has an internal partition.
\end{prop}

\begin{proof}
    Let $G=Cay(\mathbb{Z}_2 \times \mathbb{Z}_{2p};S)$, $S=\{g_1, \dots, g_k\}$. Let $T=\{(1,0);(0,p); (1,p)\}$. If $g_i \in T$, then the edges generated by $g_i$ determine a perfect matching, because they are their own negatives. If $g_i \notin T$, then the edges generated by $g_i$ determine the union of disjoint cycles. So either $|S \cap T|=3$ and $|S \setminus T|=1$ or $|S \cap T|=1$ and $|S \setminus T|=2$.
    
    Suppose that $|S \cap T| = 3$. Then for all $0 \leq q < p$ the set $\{(*,q); (*,q+~p):  * \in \mathbb{Z}_2\}$ induces a complete graph $K_4$, i.e., a 3-cohesive set. These are disjoint subgraphs for all choices of $q$, so we found two disjoint internal sets.
    
    Now suppose that $|S \cap T|=1$ and $|S \setminus T|=2$ holds. In the first case let $g_1=(1,0)$ and denote the second coordinate of $g_2$ and $g_3$ by $q$ and $r$ with $q\leq r<p$ and $ * \in \mathbb{Z}_2$. If $q=r$, we obtain again induces $K_4$ graphs in the Cayley graph thus we are done. Otherwise consider the graph $G'=Cay(\mathbb{Z}_{2p};\{q;r\})$. It is $4$-regular and it is not the complete graph $K_5$  as $G'$ has $2p$ vertices, so it has an internal partition in view of Theorem \ref{kisr}. We denote this partition  by $A' \cup B'$. Let $A=\{(*,a)\  | \ * \in \mathbb{Z}_2; a \in A'\}$ and $B=\{(*,b) \ |\  * \in \mathbb{Z}_2; b \in B'\}$. Then $A \cup B$ is an internal partition of $G$.
    
    This method works similarly in the other case $g_1\in \{(0,p);((1,p)\}$ after considering  $G'=Cay(\mathbb{Z}_{2p};\{q;r;p\})$.
\end{proof}

\begin{theorem} Every $5$-regular Cayley graph arising from an Abelian group admits an internal partition except for three graphs, described in Theorem \ref{5regcay}.
\end{theorem}

\begin{proof}
Let $\mathcal{G}$ be a finite Abelian group. Consider a $5$-regular Cayley graph $Cay(\mathcal{G}, S)$ of $\mathcal{G}$ and let $\mathcal{G}_{(2)}$ denote its subgroup generated by the elements of order at most two.\\
In the first case, suppose that  $|S \cap \mathcal{G}_{(2)}|\geq 3$. This implies we have $3$ distinct generators $g_1, g_2, g_3\in S$ of order $2$.  Then each coset of $\langle g_1, g_2, g_3 \rangle$ induces a $3$-regular subgraph on at most $8$ vertices thus we are done provided that  $|\mathcal{G}|>8$. Groups of smaller order are already considered above.\\
In the second case, we have  $|S \cap \mathcal{G}_{(2)}|< 3$, which in turn implies $|S \cap \mathcal{G}_{(2)}|=1$ by the parity of the valency of $Cay(\mathcal{G}, S)$. Then there exists $g_1\in S \setminus \mathcal{G}_{(2)}$.\\
$\langle g_1 \rangle=\mathcal{G}$ would imply that  $\mathcal{G}$ is cyclic, which case is covered already in Theorem \ref{5regcay}. Now suppose that $|\langle g_1 \rangle|=|\mathcal{G}|/2$. Then $\mathcal{G}$ must be either a cyclic group or a direct product of $\mathbb{Z}_2$ and a cyclic group. These subcases are already covered by Theorem \ref{5regcay} and \ref{5regcay2}. Finally, suppose that $|\langle g_1 \rangle|<|\mathcal{G}|/2$. The cosets of  $\langle g_1 \rangle$ determine cycles in the Cayley graph. Let us take  $g_2:=S\cap \mathcal{G}_{(2)}$-t and consider $\langle g_1, g_2\rangle$. The cosets of this subgroup induce $3$-regular graphs, moreover $|\langle g_1, g_2\rangle|\in \{|\langle g_1\rangle|, 2|\langle g_1\rangle| \}$. As a consequence, we find at least two disjoint $3$-regular subgraphs. 
\end{proof}

\section{Concluding remarks}

We presented an approach how to show the existence of cohesive sets which have rather small intersection. Although our main theorem \ref{main} does not provide a bound close enough to the desired result $o(1)$, the applied technique pinpoints several subproblems of independent interest in which any breakthrough would imply an improvement for the bound of main Theorem \ref{main} as well.

We  pose this list of problems below.

\begin{problem} Improve the bound of Ban and Linial, Lemma \ref{banlinial} by showing the existence of $\lceil r/2 \rceil$-cohesive sets in $r$-regular  $n$-vertex graphs on much less than $n/2$ vertices, subject to $n\gg r$.
\end{problem}
Note for example that if one could show the existence of a $3$-regular $H$ subgraph of the $5$-regular graph $G$ on less than $n/3=|V(G)|/3$ vertices, that would provide a straightforward  application of the Alon-Friedland-Kalai theorem.
It would be really interesting to find an analogue of the Alon--Friedland--Kalai theorem \ref{AFK} with a restriction on the size of the subgraph as well. That would enable us to easily find dense subgraphs with the prescribed maximum degree, at least for certain values of the maximum degree $d$.

In a more general form, we formalize

\begin{problem} Determine the best possible $\lambda_{r,t}$ constant, depending on $r$ and $t$, for which the following holds.
Let $G$ be a $r$-regular bipartite graph on $n$ vertices. Then there is a subgraph $H\subset G$ on at most $(\lambda_{r,t}+o(1)) n$ vertices with minimum degree $\delta(H)\geq t$.
\end{problem}

\begin{problem} Prove a common generalisation of Theorem \ref{AFK} and Theorem \ref{Lisa} which fixes the degree sequence of the subgraph and guarantee many $0$-degrees.
\end{problem}

It would be interesting to obtain a general lower bound $f(k)$ on the edge cardinality which can be guaranteed in at least one $k$-vertex subgraph of $n$-vertex graphs with a prescribed maximum degree condition. We showed that under the conditions of Proposition \ref{minfokos}, $f(k)\geq k-1$ holds for every $k\leq n$, and subsection 3.2 presents a possible way how to improve that bound at least when $k$ is not small compared to $n$.

\begin{problem} Improve and generalize the result of Proposition \ref{minfokos} by obtaining a lower bound function on the cardinality of the edges of $k$-vertex subgraphs having a given bound on the maximum degree.
\end{problem}

\begin{problem}
Prove that every Paley graph has an internal partition.
\end{problem}

\begin{problem}
Prove that almost all $5$-regular graphs have an internal partition via improving the algorithmic approach of Proposition \ref{n/2} and applying Theorem \ref{DSW}.
\end{problem}

\bibliographystyle{abbrv}
\bibliography{bib}

\end{document}